\documentclass{article}
\usepackage{amsmath,amsthm,amssymb,mathrsfs,amstext, titlesec,enumitem}
\makeatletter

\titleformat{\section}
{\normalfont\Large\bfseries}{\thesection}{1em}{}
\titleformat*{\subsection}{\Large\bfseries}
\titleformat*{\subsubsection}{\large\bfseries}
\titleformat*{\paragraph}{\large\bfseries}
\titleformat*{\subparagraph}{\large\bfseries}

\renewcommand{\@seccntformat}[1]{\csname the#1\endcsname. }

\renewenvironment{abstract}{%
    \if@twocolumn
      \section*{\abstractname}%
    \else 
      \begin{center}%
        {\bfseries \Large\abstractname\vspace{\z@}}
      \end{center}%
      \quotation
    \fi}
    {\if@twocolumn\else\endquotation\fi}
    
\makeatother
\theoremstyle{plain}
\newtheorem{thm}{Theorem}

\theoremstyle{definition}

\providecommand{\keywords}[1]{{\bf{Keywords:}} #1}

\title{An inductive proof of the Bollob\'{a}s two family theorem}

\author{Sayan Goswami \\ \textit{sayangoswami@imsc.res.in} \\ The Institute of Mathematical Sciences\\ A CI of Homi Bhabha National Institute\\ CIT Campus, Taramani, Chennai 600113, India.}

\date{\vspace{-5ex}}

\begin{document}
\maketitle
\begin{abstract}
Inspired by the inductive proof of the LYM-inequality given by P.
Frankl in \cite{key-8}, we provide an inductive proof of the Bollob\'{a}s
two family theorem \cite{key-7}.
\end{abstract}


\keywords  Bollobas two family theorem \\

The LYM inequality established by Lubell \cite{key-5}, Yamamoto \cite{key-4}
and Meshalkin \cite{key-2} is one of the fundamental result in combinatorics.
One can find several proofs of this result in literature. In \cite{key-8},
Frankl found an inductive proof of this result that uses elementary
probability theory. In \cite{key-7}, Bollob\'{a}s found a generalization
of this inequality known as Bollob\'{a}s theorem. The known proof
uses random permutation and independence of random variables. In \cite{key-6},
several results can be found towards this direction. Inspired by \cite{key-8},
here we provide a relatively elementary proof of Bollob\'{a}s theorem
that uses elementary probability theory and induction argument.
\begin{thm}
(Bollob\'{a}s two family theorem) If $m\in\mathbb{N}$ and $\mathcal{F}_{1}=\left\{ A_{1},\ldots,A_{m}\right\} $,
$\mathcal{F}_{2}=\left\{ B_{1},\ldots,B_{m}\right\} $ be two family
of sets over $X=\left\{ 1,2,\ldots,n\right\} $ such that $A_{i}\cap B_{i}=\emptyset$
and $A_{i}\cap B_{j}\neq\emptyset$ for all $i,j\in\left\{ 1,2,\ldots,m\right\} $,
then $\sum_{A_{i}\in\mathcal{F}_{1}}\frac{1}{\binom{\vert A_{i}\vert+\vert B_{i}\vert}{\vert A_{i}\vert}}\leq1.$
\end{thm}

\begin{proof}
For $n=1$, the result is true. So assume that the result is true
over any set of cardinality $n-1.$ Let $\mathcal{F}_{1}$ and $\mathcal{F}_{2}$
be two given families. For any $x\in X,$ choose $\mathcal{G}_{1}\left(x\right)=\left\{ A_{i}\in\mathcal{F}_{1}:x\notin A_{i}\,\text{and}\,x\in B_{i}\right\} .$
Note that for every pair (except the family $\mathcal{F}_{1}=\left\{ X\right\} $
and $\mathcal{F}_{2}=\left\{ \emptyset\right\} $, where the result
trivially true) of $\left(A_{i},B_{i}\right)$ one such $x\in X$
exists, infact every $x\in B_{i}$ will work. Now choose $\mathcal{G}_{2}\left(x\right)=\left\{ B_{i}\setminus\left\{ x\right\} :A_{i}\in\mathcal{G}_{1}\left(x\right)\right\} .$
Then it can be easily checked that the elements of $\mathcal{G}_{1}\left(x\right)$
and $\mathcal{G}_{2}\left(x\right)$ satisfies the condition of the
theorem over the set $X\setminus\left\{ x\right\} .$ So by induction
\[
1\geq\mathbb{E}\left(\sum_{A_{i}\in\mathcal{G}_{1}\left(x\right)}\frac{1}{\binom{\vert A_{i}\vert+\vert B_{i}\vert-1}{\vert A_{i}\vert}}\right)
\]
\[
\qquad\qquad=\sum_{A_{i}\in\mathcal{F}_{1}}\mathbb{P}\left(A_{i}\in\mathcal{G}_{1}\left(x\right)\right)\cdot\frac{1}{\binom{\vert A_{i}\vert+\vert B_{i}\vert-1}{\vert A_{i}\vert}}
\]
\[
\qquad\:\qquad\:\,\,\quad\;\:\:\,\,=\sum_{A_{i}\in\mathcal{F}_{1}}\mathbb{P}\left(x\in B_{i}\vert x\in A_{i}\cup B_{i}\right)\cdot\frac{1}{\binom{\vert A_{i}\vert+\vert B_{i}\vert-1}{\vert A_{i}\vert}}
\]
\[
\;\qquad=\sum_{A_{i}\in\mathcal{F}_{1}}\frac{\vert B_{i}\vert}{\vert A_{i}\vert+\vert B_{i}\vert}\cdot\frac{1}{\binom{\vert A_{i}\vert+\vert B_{i}\vert-1}{\vert A_{i}\vert}}
\]
\[
\qquad\qquad\quad=\sum_{A_{i}\in\mathcal{F}_{1}}\frac{\vert A_{i}\vert+\vert B_{i}\vert-\vert A_{i}\vert}{\vert A_{i}\vert+\vert B_{i}\vert}\cdot\frac{1}{\binom{\vert A_{i}\vert+\vert B_{i}\vert-1}{\vert A_{i}\vert}}
\]
\[
\!\!\!\!\!\!\!\!\!\!\!\!\!\!\!\!\!\!\!\!\!\!\!\!=\sum_{A_{i}\in\mathcal{F}_{1}}\frac{1}{\binom{\vert A_{i}\vert+\vert B_{i}\vert}{\vert A_{i}\vert}}.
\]
This completes the proof.
\end{proof}


\begin{thebibliography}{1}
\bibitem{key-6}B. Bollob\'{a}s, Combinatorics. Set systems, hypergraphs,
families of vectors and combinatorial probability. Cambridge University
Press, Cambridge, 1986.

\bibitem{key-7} B. Bollob\'{a}s, On generalized graph, Acta Math.
Acad. Sci. Hungar. 16, 447-452, 1965.

\bibitem{key-8} P. Frankl, A probabilistic proof for the lym-inequality,
Dicrete Mathematics. Volume 43, Issues 2--3, 1983, Page 325.

\bibitem{key-5} D. Lubell, A short proof of Sperner\textquoteright s
theorem, J. Cambin. Theory 1 (1966) 299.

\bibitem{key-2} L.D. Meshalkin, A generalization of Sperner\textquoteright s
theorem on the number of subsets of a finite set, Theor. Probability
Appl. 8 (1963) 203-204.

\bibitem{key-3} E. Sperner, Ein Satz iiber Untermengen einer endlichen
Menge, Math. Z. 27 (1928) 544-548.

\bibitem{key-4} K. Yamamoto, Logarithmic order of free distributive
lattices, J. Math. Sot. Japan 6 (1954) 343-353.
\end{thebibliography}
\end{document}